\documentclass[a4paper,12pt]{amsart}
\usepackage{amssymb}
\usepackage{ifthen}
\usepackage{tikz}
\usetikzlibrary{arrows.meta}
\usetikzlibrary{calc}
\usepackage{graphicx}
\usepackage{mathrsfs}
\nonstopmode

\usepackage{epsfig}
\usepackage{pstricks}
\usepackage{pst-node}
\usepackage{pst-tree}
\usepackage{pst-plot}
\usepackage{pst-text}
\usepackage{multido}
\usepackage{subfigure}

\numberwithin{equation}{section}
\newtheorem{thm}{Theorem}[section]

\newtheorem{lem}[thm]{Lemma}

\theoremstyle{definition}

\newenvironment{pf}[1][]{%
 \vskip 3mm
 \noindent
 \ifthenelse{\equal{#1}{}}%
  {{\slshape Proof. }}%
  {{\slshape #1.} }%
 }%
{\qed\bigskip}

\newcounter{alphabet}
\newcounter{tmp}
\newenvironment{Thm}[1][]{\refstepcounter{alphabet}%
\bigskip%
\noindent%
{\bf Theorem \Alph{alphabet}}%
\ifthenelse{\equal{#1}{}}{}{ (#1)}%
{\bf .}
\itshape}{\vskip 8pt}

\newcommand{\C}{{\mathbb C}}
\newcommand{\D}{{\mathbb D}}

\newcommand{\uhp}{{\mathbb H}}

\newcommand{\N}{{\mathbb N}}

\newcommand{\R}{{\mathbb R}}

\newcommand{\Z}{{\mathbb Z}}

\renewcommand{\Im}{\,{\operatorname{Im}\,}}

\renewcommand{\Re}{\,{\operatorname{Re}\,}}

\newcommand{\inv}{^{-1}}
\newcommand{\area}{{\operatorname{Area}}}

\renewcommand{\arg}{\,{\operatorname{arg}\,}}

\newcommand{\aand}{{\quad\text{and}\quad}}

\newcommand{\arth}{{\operatorname{arth}\,}}

\begin{document}
\bibliographystyle{amsplain}
\title{
Asymptotic perimeter estimates for spirallike functions
}

\author[Z. Alsharari]{Zain Shayim Alsharari}
\address{Graduate School of Information Sciences,
Tohoku University, Aoba-ku, Sendai 980-8579, Japan}
\email{zain.shayim.alsharari.q7@dc.tohoku.ac.jp, zainjapan11@gmail.com }
\author[T. Sugawa]{Toshiyuki Sugawa}
\email{sugawa@tohoku.ac.jp}
\begin{abstract}
By a result of F.~R.~Keogh in 1959, we have the inequality $L(r,f)/M(r,f)\le 2\pi+4\log[(1+r)/(1-r)]$
for a starlike function $f$ on the unit disk, where $L(r,f)$ and $M(r,f)$
denote the length of the curve $\theta\mapsto f(re^{i\theta})$ and the maximum value
of $|f(z)|$ on the circle $|z|=r$ for $0<r<1.$
E.~Crane and D.~Markose in their 2005 paper showed that the constant $4$ is
best possible.
D.~K.~Thomas obtained in 1968 the inequality $L(r,f)/\sqrt{A(r,f)}\le 2\sqrt{\pi}(1+\log[(1+r)/(1-r)])$
for a starlike function $f$, where $A(r,f)$ is the area of the image of the disk $|z|<r$
under the mapping $f.$
In this paper, we extend these results to the case of $\lambda$-spirallike functions
for a given $\lambda$ with $-\pi/2<\lambda<\pi/2.$
\end{abstract}
\thanks{
The present work was supported in part by JSPS KAKENHI Grant Number JP25K07022.
}
\dedicatory{Dedicated to Professor Ponnusamy on the occasion of his 65th birth anniversary}
\maketitle
\section{Introduction}

For a given $\lambda\in(-\pi/2,\pi/2),$ we consider the $\lambda$-logarithmic spiral
$$
\gamma_0(t)=\exp(t e^{i\lambda}),\quad t\in\R.
$$
For a point $w\in\C,$ the arc $\{w\gamma_0(t): -\infty<t\le 0\}\cup\{0\}$
is called the $\lambda$-spiral segment joining $w$ and $0$ and will be denoted by
$[0,w]_\lambda.$
A domain $\Omega$ with $0\in\Omega\subset\C$ is called $\lambda$-spirallike
with respect to the origin
if $[0,w]_\lambda\subset\Omega$ whenever $w\in\Omega.$
A holomorphic function $f$ on the unit disk $\D=\{z\in\C: |z|<1\}$ is called
$\lambda$-spirallike if $f(0)=0$ and if
$f$ maps $\D$ univalently onto a $\lambda$-spirallike
domain with respect to the origin.
When $\lambda=0,$ $0$-spiral segment $[0,w]_0$ is nothing but the Euclidean line segment
$[0,w]$ and therefore $0$-spirallike means starlike.
It is well known that a holomorphic function $f$ on $\D$ with $f(0)=0$
is $\lambda$-spirallike if and only if the inequality
$$
\Re\left(e^{-i\lambda}\frac{zf'(z)}{f(z)}\right)>0
$$
holds for every $z\in\D^*=\D\setminus\{0\}$ (see, for instance, Duren \cite{Duren:univ}).
This condition means that the $\lambda$-argument (see \cite{KS12spl} for details)
$$
\arg_\lambda(w)=\arg w-(\tan\lambda)\log|w|
$$
is increasing along the image curve $f(re^{it})$ in $0\le t<2\pi$ for each $r\in(0,1)$ (see \eqref{eq:arg} below).
Note that $\arg_\lambda(\gamma_0(t))=0$ for the curve $\gamma_0$ defined above.

In general, for a holomorphic map $f$ on $\D,$ we define
$$
L(r,f)=\int_{|z|=r}|f'(z)||dz|=\int_0^{2\pi}r|f'(re^{i\theta})|d\theta
\aand
M(r,f)=\max_{|z|=r} |f(z)|.
$$
Keogh \cite{Keogh59} showed the following result when $f$ is bounded but
his proof is readily extended to the general case as follows (see \cite[Theorem 4.1.13]{TTV:univ}).

\begin{Thm}
Let $f$ be a starlike function on $\D.$
Then for $0<r<1$
$$
L(r,f)\le M(r,f)\left(2\pi+4\log\frac{1+r}{1-r}\right).
$$
\end{Thm}

Keogh asked in the same paper \cite{Keogh59} whether there is a bounded starlike function $f$
such that $L(r,f)>\alpha\log\frac{1+r}{1-r}$ for some $\alpha>0.$
Subsequently, Hayman \cite{Hayman61} gave an example of a bounded starlike function $f$ with
$\limsup_{r\to 1}L(r,f)/\log\frac{1+r}{1-r}>0.$
Much later, Lewis \cite{Lewis76} improved it by replacing $\limsup$ by $\liminf.$
Moreover, Crane and Markose \cite{CM05} showed that the constant $4$ in Keogh's theorem
is the best possible.
Our first result here is to extend Theorem A to spirallike functions.

\begin{thm}\label{thm:main1}
Let $f$ be a $\lambda$-spirallike function on $\D$ for some $\lambda\in(-\pi/2,\pi/2).$
Then for $r\in(0,1),$
\begin{equation}\label{eq:spiral}
L(r,f)\le M(r,f)\left(2\sqrt2\pi+4\cos\lambda\cdot\log\frac{1+r}{1-r}\right).
\end{equation}
\end{thm}

This will be proved in the next section using the same method as Keogh \cite{Keogh59}.
Our main contribution in this paper is to verify the optimality of
the constant $4\cos\lambda$ in \eqref{eq:spiral} by emplyoing
methods from Crane and Markose \cite{CM05}

\begin{thm}\label{thm:main2}
Let $\lambda\in(-\pi/2,\pi/2).$
For every $\varepsilon\in(0,4),$ there exists a $\lambda$-spirallike function
$f$ on $\D$ such that
$$
\liminf_{r\to 1^-}\frac{L(r,f)}{M(r,f)\log\frac{1+r}{1-r}}\ge (4-\varepsilon)\cos\lambda.
$$
\end{thm}

We mention here related estimates involving the growth of the image area:
$$
A(r,f)=\iint_{|z|\le r}|f'(z)|^2dxdy,\quad z=x+iy.
$$
Pommerenke proved the inequality (see \cite[Lemma 2]{Pom65AM})
$$
M(r,f)\le 4\sqrt{\frac{A(r,f)}{\pi}\log\frac3{1-r}},\quad 0<r<1,
$$
for an arbitrary univalent analytic function $f$ on $\D$ with $f(0)=0.$ 
Therefore, in conjunction with Theorem A,
$L(r,f)$ is estimated from above in terms of $A(r,f)$ for a starlike function $f.$
Indeed, Thomas \cite{Thomas68} showed a better inequality.

\begin{Thm}
Let $f$ be a starlike function on $\D.$
Then
$$
2\sqrt{\pi A(r,f)}\le L(r,f)\le2\sqrt{\pi A(r,f)}\left(1+\log\frac{1+r}{1-r}\right),\quad 0<r<1.
$$
\end{Thm}

Earlier, Thomas \cite{Thomas67} proved the weaker estimate
$L(r,f)\le 8A(1+\log\frac1{1-r^2})$ when $A(r,f)$ is bounded by a constant $A.$
In the Mathematical Review \cite{Rob67MR} of \cite{Thomas67}, Robertson suggested a way to improve the
constant $8A$ to $14A/3$ and pointed out that the same method yields the inequality
$L(r,f)\le 8A(1+\cos\lambda\log\frac1{1-r^2})$ for a $\lambda$-spirallike function $f.$
Similarly, we will show the following extension of Theorem B to spirallike functions.

\begin{thm}\label{thm:main3}
Let $\lambda\in(-\pi/2,\pi/2).$
For a $\lambda$-spirallike function $f$ on $\D,$ the inequalities
$$
2\sqrt{\pi A(r,f)}\le L(r,f)\le \sqrt{A(r,f)}
\left(2\sqrt2\cos\lambda\,\log\frac{1+r}{1-r}+\sqrt 2\pi\right)
$$
hold for each $0<r<1.$
\end{thm}

In view of the proof below, it seems that the coefficient $2\sqrt2\cos\lambda$ is not best possible.
On the other hand, the example in Section 3 implies that, for any $\varepsilon>0,$ there
is a $\lambda$-spirallike function $f$ with $A(r,f)\to \pi$ as $r\to1^-$ and
$$
\liminf_{r\to 1^-}\frac{L(r,f)}{\sqrt{A(r,f)}\log\frac{1+r}{1-r}}>
\frac{4-\varepsilon}{\sqrt{\pi}}\cos\lambda.
$$
Note that $4/\sqrt\pi=2.256\dots.$
Determining the optimal coefficient remains an open problem.

As is well known, the Koebe function $K(z)=z/(1-z)^2$ is extremal in various problems.
Therefore, it is worthwhile to examine it to our estimates.
It is known that $L(r,K)=O((1-r)^{-2})$ as $r\to 1^-$ (see \cite[p.~232]{Duren:univ}).
On the other hand, $A(r,K)=\pi r^2(1+4r^2+r^4)/(1-r^2)^4$ is asymptotically equal
to $\frac{3\pi}{8}(1-r)^{-4}$ as $r\to 1^-.$
Hence, $L(r,K)/\sqrt{A(r,K)}=O(1)=o(-\log(1-r))$ as $r\to1^-.$
In our results, the Koebe function is not extremal in the asymptotic sense.

\section{Proof of main inequalities}

\begin{pf}[Proof of Theorem \ref{thm:main1}]
Let $f$ be a $\lambda$-spirallike function on $\D.$
Then
$$
e^{-i\lambda}\frac{zf'(z)}{f(z)}=p(z)\cos\lambda-i\sin\lambda
$$
for some analytic function $p$ on $\D$ with $\Re p>0$ and $p(0)=1.$
The condition for $p$ may be rephrased to that $p(z)$ is subordinate
to the function $p_0(z)=\frac{1+z}{1-z}.$
The inequality
$$
\int_0^{2\pi} |p_0(re^{i\theta})|d\theta\le 2\pi+4\log\frac{1+r}{1-r}
$$
is known to hold for $0<r<1$ (see \cite[Theorem 1]{Keogh59}).
By Littlewood's subordination theorem \cite[Theorem 1.7]{Duren:hp}
\begin{align*}
L(r,f)&=\int_{|z|=r}|f'(z)||dz|
=\int_{|z|=r}|f(z)||p(z)\cos\lambda-i\sin\lambda|\frac{|dz|}{|z|} \\
&\le M(r,f)\left(\cos\lambda\int_0^{2\pi} |p(re^{i\theta})|d\theta
+2\pi\sin|\lambda|\right) \\
&\le M(r,f)\left(\cos\lambda\int_0^{2\pi} |p_0(re^{i\theta})|d\theta
+2\pi\sin|\lambda|\right) \\
&\le M(r,f)\left(\cos\lambda\left[2\pi+4\log\frac{1+r}{1-r}\right]
+2\pi\sin|\lambda|\right).
\end{align*}
Since $\cos\lambda+\sin|\lambda|\le \sqrt2,$ the inequality \eqref{eq:spiral} follows.
\end{pf}

In order to prove Thereom \ref{thm:main3}, we prepare the following lemma, which generalizes the well-known area formula
$\area(D)=\frac12\int_0^{2\pi}R^2d\theta$ in terms of polar coordinates.

\begin{lem}\label{lem:area}
Let $\lambda\in(-\pi/2,\pi/2).$
Suppose that $D$ is a $\lambda$-spirallike domain with respect to the origin in $\C$ 
bounded by a Jordan curve parametrized as $|w|=R(\!\arg_\lambda w)$ for a $C^1$ function
$R(\phi)>0$ on $[0,2\pi].$
Then the Euclidean area of $D$ is given by
$$
\area(D)=\frac12\int_0^{2\pi}R(\phi)^2d\phi.
$$
\end{lem}

\begin{pf}
We write $|w|=R$ and $\log w=\rho+i\theta.$
Then, by definition, $\phi=\arg_\lambda w=\theta-\rho\tan\lambda.$
Since $R=e^\rho,$ we compute
\begin{align*}
2\,\area(D)&=\int_{\partial D}R^2d\theta \\
&=\int_{\partial D}R^2(d\phi+\tan\lambda \,d\rho) \\
&=\int_{\partial D}\left[R^2d\phi+\frac{\tan\lambda}2 d(R^2)\right] \\
&=\int_{\partial D}R^2d\phi
=\int_0^{2\pi}R(\phi)^2d\phi.
\end{align*}
\end{pf}

The following inequality is a slight extension of a well-known inequality
for Carath\'eodory functions.

\begin{lem}\label{lem:q}
Let $q$ be an analytic function on $\D$ with $\Re q>0.$
Then
$$
|q'(z)|\le \frac{2\Re q(z)}{1-|z|^2},\quad z\in\D.
$$
\end{lem}

\begin{pf}
Let $q(0)=a+ib$ with $a>0.$
Then there is a Carath\'eodory function $p$ (that is, $\Re p>0$ and $p(0)=1$)
such that $q=ap+ib.$
Since $|p'(z)|\le 2\Re p(z)/(1-|z|^2)$ (see \cite[Theorem 3.1.1]{TTV:univ}),
we have
$$
|q'(z)|=a|p'(z)|\le \frac{2a\Re p(z)}{1-|z|^2}=\frac{2\Re q(z)}{1-|z|^2}.
$$
\end{pf}

We next prove Theorem \ref{thm:main3} by following the arguments in \cite{Thomas68}.

\begin{pf}[Proof of Theorem \ref{thm:main3}]
The left-hand inequality is nothing but the classical isoperimetric inequality.
It remains to show the right-hand inequality.
Let $q(z)=e^{-i\lambda}zf'(z)/f(z).$
Then, by a straightforward calculation, we have
\begin{align}\label{eq:arg}
\frac{\partial}{\partial\theta} \arg_\lambda f(re^{i\theta})
&=\frac{\partial}{\partial\theta} \arg f(re^{i\theta})
-\tan\lambda\,\frac{\partial}{\partial\theta}\log |f(re^{i\theta})| \\
\notag
&=\frac{\partial}{\partial\theta} \Im\log f(re^{i\theta})
-\tan\lambda\,\frac{\partial}{\partial\theta}\Re\log f(re^{i\theta}) \\
\notag
&=\frac{1}{\cos\lambda}\Im \frac{\partial}{\partial\theta}\Big[ e^{-i\lambda}\log f(re^{i\theta})\Big] \\
\notag
&=\frac{\Re q(re^{i\theta})}{\cos\lambda}>0.
\end{align}
With the aid of the relation $|zf'(z)|=|f(z)q(z)|,$ the Cauchy-Schwarz inequality
implies
$$
L(r,f)^2=\left(\int_{|z|=r}|f(z)q(z)|\frac{|dz|}{|z|}\right)^2
\le \int_0^{2\pi}|f(z)^2q(z)|d\theta\cdot\int_0^{2\pi}|q(z)|d\theta
$$
for $z=re^{i\theta}.$
By the same estimation as in the proof of Theorem \ref{thm:main1}, we have
$$
\int_0^{2\pi}|q(z)|d\theta\le 
2\pi(\cos\lambda+\sin|\lambda|)+4\cos\lambda\log\frac{1+r}{1-r}
\le 2\sqrt2\pi+4\cos\lambda\log\frac{1+r}{1-r}.
$$
Next we estimate the term
$$
I(r)=\int_0^{2\pi}|f(z)^2q(z)|d\theta,\quad z=re^{i\theta}.
$$
Since $I(0)=0$ and
\begin{align*}
I'(r)
&=\int_0^{2\pi}\left(2\Re[e^{i\theta}\overline{f(z)}f'(z)]|q(z)|+\frac{\Re[e^{i\theta}\overline{q(z)}q'(z)]}{|q(z)|}|f(z)|^2\right) d\theta \\
&\le\int_0^{2\pi}\left(2|f(z)f'(z)q(z)|+|q'(z)f(z)^2|\right) d\theta \\
&=2\int_0^{2\pi}|f'(z)|^2|z|d\theta+\int_0^{2\pi}|q'(z)f(z)^2| d\theta,
\end{align*}
we obtain
$$
I(r)\le 2\int_0^r\int_0^{2\pi}|f'(te^{i\theta})|^2td\theta dt+
\int_0^r\int_0^{2\pi}|q'(te^{i\theta})f(te^{i\theta})^2|d\theta dt
=2A(r,f)+J(r).
$$
By Lemma \ref{lem:q}, we estimate the second term as
$$
J(r)\le 2\int_0^r\int_0^{2\pi}|f(te^{i\theta})|^2
\Re q(te^{i\theta})d\theta \frac{dt}{1-t^2}.
$$
If we write $\phi=\arg_\lambda f(te^{i\theta})$ and $R_t=R_t(\phi)=|f(te^{i\theta})|$
for a fixed $t,$ by \eqref{eq:arg}, we further estimate $J(r):$
$$
J(r)\le
2\cos\lambda\int_0^r\int_0^{2\pi}R_t^2\frac{\partial \phi}{\partial\theta}d\theta \frac{dt}{1-t^2}
=2\cos\lambda\int_0^r\int_0^{2\pi}R_t(\phi)^2d\phi \frac{dt}{1-t^2}.
$$
We now apply Lemma \ref{lem:area} to obtain
$$
\frac{J(r)}{\cos\lambda}\le 4\int_0^r A(t,f)\frac{dt}{1-t^2}
\le 4 A(r,f)\int_0^r \frac{dt}{1-t^2}
=2A(r,f)\log\frac{1+r}{1-r}.
$$
Therefore, we finally get
$$
L(r,f)^2\le 2A(r,f)\left(1+\cos\lambda\log\frac{1+r}{1-r}\right)\cdot
\left(2\sqrt2\pi+4\cos\lambda\log\frac{1+r}{1-r}\right).
$$
Since $\sqrt 2 \pi=4.44\dots>2,$ we obtain the required inequality.
\end{pf}

\section{Construction}

Let $\lambda\in(-\pi/2,\pi/2)$ be given and set $\mu=ie^{i\lambda}.$
Fix an integer $N\ge 3$ and a real number $T$ such that $T(1-N\inv)\ge 1/2.$
Let $U$ be the domain defined by
$$
U=\uhp\setminus \bigcup_{n=0}^\infty A_n,
$$
where 
$$
A_0=\bigcup_{m\in\Z}(m,\alpha_{0,m}],
\quad
A_n=\bigcup_{m\in\Z-N\Z}(mN^{-n},\alpha_{n,m}] \quad (n\ge 1)
$$
and
$$
\alpha_{n,m}=\frac{m+T\mu}{N^n}.
$$
The function $\Pi(\zeta)=e^{2\pi i\zeta}$ is a holomorphic universal covering
projection of $\uhp$ onto the punctured disk $\D^*=\D\setminus\{0\}.$
We set $\Omega^*=\Pi(U)$ and $\Omega=\Omega^*\cup\{0\}.$
Then $\Omega$ is a simply connected domain obtained by removing the
$\lambda$-spiral segment joining $a_{n,m}=\Pi(\alpha_{n,m})\in\D$ and
$a_{n,m}'=\Pi(mN^{-n})\in\partial\D$ for $n\in\N_0=\N\cup\{0\}$ and $1\le m\le N^n$ with $N\nmid m.$
Note that $\arg_\lambda(a_{n,m})=2\pi m/N^n$ and that $\Omega$ is $\lambda$-spirallike
with respect to the origin.
This construction agrees with that of Crane and Markose \cite{CM05}
when $\lambda=0.$

\begin{figure}[htbp]
\centering
\definecolor{mathorange}{RGB}{255,120,0}

\begin{tikzpicture}[scale=4.35, line cap=round, line join=round]

\def\sinlam{0.5}
\def\coslam{0.8660254038}

\begin{scope}[yshift=0cm]

\node at (0.5,1.18) {\Huge $U$};

\draw[thick] (-0.53,0) -- (1.27,0);
\node[below] at (0,-0.02) {\Huge $0$};
\node[below] at (1,-0.02) {\Huge $1$};

\foreach \n/\m/\rr in {
0/0/0.010,
0/1/0.010,
1/-1/0.007,
1/1/0.007,
1/2/0.007}
{
  \pgfmathsetmacro{\d}{pow(3,\n)}
  \pgfmathsetmacro{\xb}{\m/\d}
  \pgfmathsetmacro{\xa}{(\m-\sinlam)/\d}
  \pgfmathsetmacro{\ya}{\coslam/\d}
  \draw[very thick] (\xb,0) -- (\xa,\ya);
  \fill[black] (\xa,\ya) circle (\rr);
}

\foreach \m in {-4,-2,-1,1,2,4,5,7,8,10,11}
{
  \pgfmathsetmacro{\xb}{\m/9}
  \pgfmathsetmacro{\xa}{(\m-\sinlam)/9}
  \pgfmathsetmacro{\ya}{\coslam/9}

  \draw[thin, black!75] (\xb,0) -- (\xa,\ya);
  \fill[black] (\xa,\ya) circle (0.0045);
}

\foreach \m in {-14,-13,-11,-10,-8,-7,-5,-4,-2,-1,1,2,4,5,7,8,
10,11,13,14,16,17,19,20,22,23,25,26,28,29,31,32,34}
{
  \pgfmathsetmacro{\xb}{\m/27}
  \pgfmathsetmacro{\xa}{(\m-\sinlam)/27}
  \pgfmathsetmacro{\ya}{\coslam/27}

  \draw[thin, black!25] (\xb,0) -- (\xa,\ya);
  \fill[black] (\xa,\ya) circle (0.0025);
}

\fill[mathorange] (0,0.8660254038) circle (0.010);
\fill[mathorange] (-0.013333,0.2886751346) circle (0.008);
\fill[mathorange] (0.333333,0.2886751346) circle (0.008);
\fill[mathorange] (0.680000,0.2886751346) circle (0.008);

\foreach \m in {-4,...,11}
{
  \pgfmathsetmacro{\xz}{\m/9}
  \pgfmathsetmacro{\yz}{\coslam/9}
  \fill[mathorange] (\xz,\yz) circle (0.0048);
}

\foreach \m in {-14,...,34}
{
  \pgfmathsetmacro{\xz}{\m/27}
  \pgfmathsetmacro{\yz}{\coslam/27}
  \fill[mathorange] (\xz,\yz) circle (0.0035);
}

\node[above] at (-0.50,0.90) {\Large $\alpha_{0,0}$};
\node[above] at (0.50,0.90) {\Large $\alpha_{0,1}$};

\node[mathorange] at (0.00,0.98) {\large $\zeta_{0,0}$};
\node[left] at (-0.44,0.36) {\normalsize $\alpha_{1,-1}$};
\node[mathorange] at (-0.02,0.36) {\normalsize $\zeta_{1,0}$};

\node[black] at (0.14,0.36) {\normalsize $\alpha_{1,1}$};
\node[mathorange] at (0.34,0.36) {\normalsize $\zeta_{1,1}$};

\node[black] at (0.49,0.36) {\normalsize $\alpha_{1,2}$};
\node[mathorange] at (0.67,0.36) {\normalsize $\zeta_{1,2}$};
\end{scope}

\draw[
  ultra thick,
  -{Latex[length=5.5mm,width=4.5mm]}
]
(0.42,-0.16) -- (0.42,-0.66);

\node[right] at (0.46,-0.40) {\Huge $\Pi$};

\begin{scope}[shift={(0.32,-1.72)}, scale=0.90]

\draw[very thick] (0,0) circle (1);
\node at (1.25,-0.08) {\Huge $\Omega$};

\draw[very thick, domain=0:1, samples=80, variable=\s]
plot ({exp(-5.44139809*\s)*cos(360*(0-0.5*\s))},
      {exp(-5.44139809*\s)*sin(360*(0-0.5*\s))});

\foreach \m in {1,2}
{
  \draw[very thick, domain=0:1, samples=80, variable=\s]
  plot ({exp(-1.81379936*\s)*cos(120*(\m-0.5*\s))},
        {exp(-1.81379936*\s)*sin(120*(\m-0.5*\s))});
}

\foreach \m in {1,2,4,5,7,8}
{
  \draw[thin, black!45, domain=0:1, samples=35, variable=\s]
  plot ({exp(-0.60459979*\s)*cos(40*(\m-0.5*\s))},
        {exp(-0.60459979*\s)*sin(40*(\m-0.5*\s))});
}

\foreach \m in {1,2,4,5,7,8,10,11,13,14,16,17,19,20,22,23,25,26}
{
  \draw[thin, black!35, domain=0:1, samples=20, variable=\s]
  plot ({exp(-0.20153326*\s)*cos(13.333333*(\m-0.5*\s))},
        {exp(-0.20153326*\s)*sin(13.333333*(\m-0.5*\s))});
}

\pgfmathsetmacro{\x}{exp(-5.44139809)*cos(-180)}
\pgfmathsetmacro{\y}{exp(-5.44139809)*sin(-180)}
\fill[black] (\x,\y) circle (0.020);

\foreach \m in {1,2}
{
  \pgfmathsetmacro{\x}{exp(-1.81379936)*cos(120*(\m-0.5))}
  \pgfmathsetmacro{\y}{exp(-1.81379936)*sin(120*(\m-0.5))}
  \fill[black] (\x,\y) circle (0.018);
}

\foreach \m in {1,2,4,5,7,8}
{
  \pgfmathsetmacro{\x}{exp(-0.60459979)*cos(40*(\m-0.5))}
  \pgfmathsetmacro{\y}{exp(-0.60459979)*sin(40*(\m-0.5))}
  \fill[black] (\x,\y) circle (0.013);
}

\foreach \m in {1,2,4,5,7,8,10,11,13,14,16,17,19,20,22,23,25,26}
{
  \pgfmathsetmacro{\x}{exp(-0.20153326)*cos(13.333333*(\m-0.5))}
  \pgfmathsetmacro{\y}{exp(-0.20153326)*sin(13.333333*(\m-0.5))}
  \fill[black] (\x,\y) circle (0.008);
}

\pgfmathsetmacro{\x}{exp(-5.44139809)*cos(180)+0.05}
\pgfmathsetmacro{\y}{exp(-5.44139809)*sin(180)}
\fill[mathorange] (\x,\y) circle (0.020);

\foreach \m in {1,2,3}
{
  \pgfmathsetmacro{\x}{exp(-1.81379936)*cos(120*\m)}
  \pgfmathsetmacro{\y}{exp(-1.81379936)*sin(120*\m)}
  \fill[mathorange] (\x,\y) circle (0.018);
}

\foreach \m in {1,...,9}
{
  \pgfmathsetmacro{\x}{exp(-0.60459979)*cos(40*\m)}
  \pgfmathsetmacro{\y}{exp(-0.60459979)*sin(40*\m)}
  \fill[mathorange] (\x,\y) circle (0.013);
}

\foreach \m in {1,...,27}
{
  \pgfmathsetmacro{\x}{exp(-0.20153326)*cos(13.333333*\m)}
  \pgfmathsetmacro{\y}{exp(-0.20153326)*sin(13.333333*\m)}
  \fill[mathorange] (\x,\y) circle (0.008);
}

\node[black] at (-0.09,-0.06) {\small $a_{0,0}$};
\node[black] at (0.21,0.20)  {\small $a_{1,1}$};
\node[black] at (-0.27,0.06) {\small $a_{1,2}$};

\node[mathorange] at (0.13,0.06) {\small $w_{0,0}$};
\node[mathorange] at (-0.11,0.24) {\small $w_{1,1}$};
\node[mathorange] at (-0.11,-0.23) {\small $w_{1,2}$};
\node[mathorange] at (0.31,-0.01) {\small $w_{1,0}$};
\end{scope}

\end{tikzpicture}

\caption{The domains $U$ and $\Omega$ and the covering projection
$\Pi(\zeta)=e^{2\pi i\zeta}$.}
\label{fig:construction}
\end{figure}
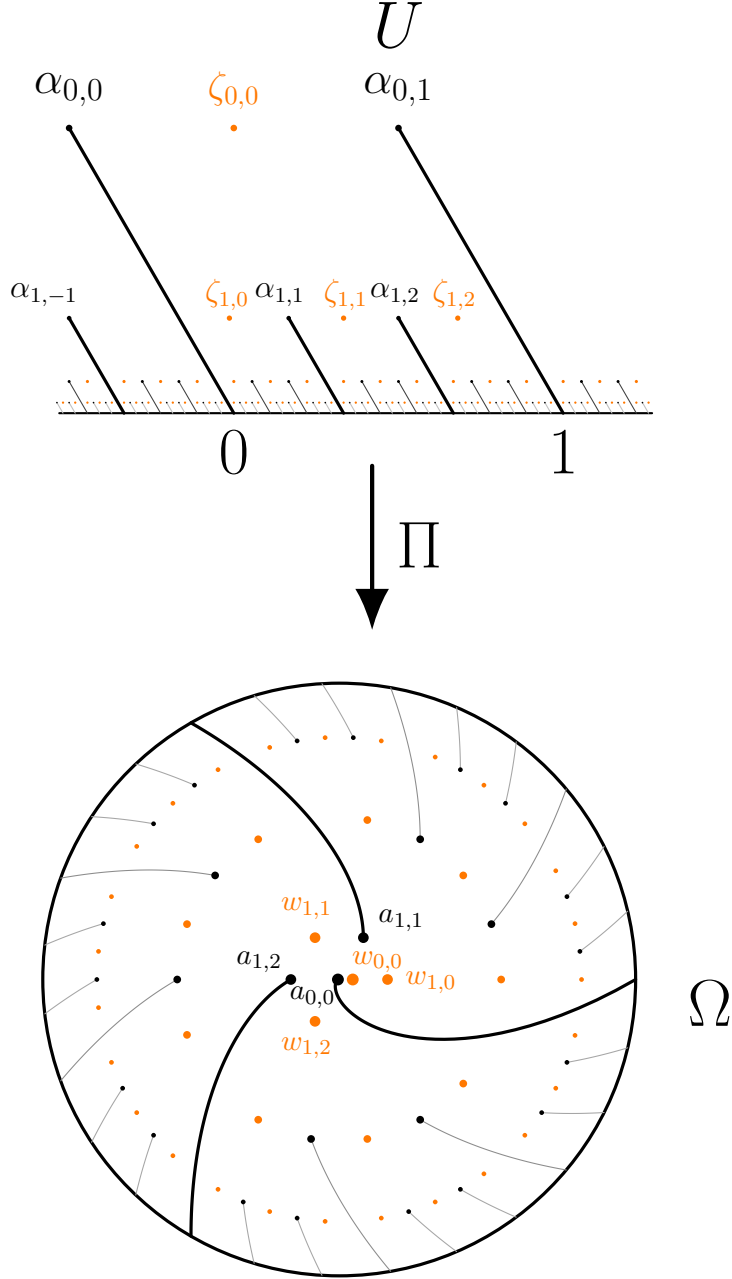
Let $f:\D\to \Omega$ be a conformal homeomorphism with $f(0)=0$ and $f'(0)>0.$
Then, by construction, $f$ is a (bounded) $\lambda$-spirallike function.
The restriction $f:\D^*\to\Omega^*$ lifts to a conformal homeomorphism $\tilde f:\uhp\to U$
via the covering map $\Pi$ so that $f\circ\Pi=\Pi\circ\tilde f$ on $\uhp.$
We consider the sets
$$
\tilde\Gamma_n=\{\zeta_{n,m}: m\in\Z\}, \quad
\zeta_{n,m}=\frac{m+1/2+T\mu}{N^n}
$$
and $\Gamma_n=\Pi(\tilde\Gamma_n)=\{w_{n,m}: m=0,1,2,\dots, N^n-1\},$
where $w_{n,m}=\exp(2\pi i\zeta_{n,m}).$
Note that $\Gamma_n\subset\Omega$ and that $\arg_\lambda (w_{n,m})=2\pi(m+1/2)/N^n.$
Let $z_{n,m}=f\inv(w_{n,m})$ and set 
$$
r_n=\max\{|z_{n,m}|: m=0,1,2,\dots, N^n-1\}.
$$
Then the image of the circle $C_r=\{z: |z|=r\}$ under $f$ is a Jordan curve in $\Omega$,
the inside of which contains all the points $w_{n,m}$ for $r_n<r<1.$

We will show, as in \cite{CM05}, the following two results.

\begin{lem}\label{lem:1}
For the above function $f,$ the inequality
$$
L(r,f)\ge 4\pi T(1-N\inv)n-c_1(T)
$$
holds for $r_n<r<1,$
where $c_1(T)$ is a constant depending only on $T.$
\end{lem}

\begin{pf}
Fix $N$ and $n$ and $r\in(r_n,1).$
Then we define continuous functions $R$ and $\Theta$ on $\R$ by
$$
R(t)=|f(re^{it})|
\aand
\Psi(t)=\arg_\lambda (f(re^{it}))=\arg f(re^{it})-(\tan\lambda) \log R(t).
$$
Since $f$ is $\lambda$-spirallike, the function $\Psi$ is increasing and satisfies
$\Psi(t+2\pi)=\Psi(t)+2\pi$ for $t\in\R.$
Then the curve $\gamma(t)=f(re^{it})$ is expressed by
$$
\gamma(t)=\exp\left((1+i\tan\lambda)\log R(t)+i\Psi(t)\right).
$$
Then we compute
$$
\gamma'(t)=\gamma(t)\left((1+i\tan\lambda)\frac{R'(t)}{R(t)}+i\Psi'(t)\right)
$$
and, since $R(t)<1,$ we obtain
\begin{equation}\label{eq:gamma}
|\gamma'(t)|\ge |1+i\tan\lambda||R'(t)|-\Psi'(t)R(t)
\ge \frac{|R'(t)|}{\cos\lambda}-\Psi'(t).
\end{equation}
Let $t_k~ (k=1,2,\dots, 2N^n)$ be real numbers with
$\Psi(t_k)=\pi k/N^n$ and $t_1<t_2<\dots <t_{2N^n}<t_1+2\pi.$
Put $t_{2N^n+1}=t_{1}+2\pi.$

For each integer $j$ with $0\le j\le n-1$ we denote by $I_j$ the set of
integers $m$ with $m\in N^j\Z-N^{j+1}\Z$ and $1\le m\le N^n.$
Any $m\in I_j$ can be expressed as $m=N^jm'$ for some integer
$m'$ with $1\le m'\le N^{n-j}$ and $N\nmid m'.$
We now observe that the cardinality of $I_j$ is $N^{n-j}-N^{n-j-1}
=N^{n-j}(1-N\inv).$
For each $m\in I_j,$ we have
$\Psi(t_{2m})=2\pi m/N^n=2\pi m'/N^{n-j}=\arg_\lambda(a_{n-j,m'})$
and thus
$$
R(t_{2m})<|a_{n-j,m'}|=\exp(-2\pi TN^{j-n}\cos\lambda).
$$
On the other hand, since the image curve $f(C_r)$ separates
the points $w_{m,n}$ from $\partial\Omega,$ we have
$$
R(t_{2m+1})>|w_{m,n}|=\exp(-2\pi TN^{-n}\cos\lambda)
\aand
R(t_{2m-1})>|w_{m-1,n}|=\exp(-2\pi TN^{-n}\cos\lambda)
$$
for $m\in I_j.$
Hence, by \eqref{eq:gamma} and the previous inequalities,
\begin{align*}
& ~ \quad \int_{t_{2m-1}}^{t_{2m+1}}|\gamma'(t)|dt \\
&\ge\int_{t_{2m-1}}^{t_{2m+1}} 
\left(\frac{|R'(t)|}{\cos\lambda}-\Psi'(t)\right)dt \\
&\ge \frac{R(t_{2m+1})-R(t_{2m})+R(t_{2m-1})-R(t_{2m})}{\cos\lambda}-\Psi(t_{2m+1})+\Psi(t_{2m-1}) \\
&\ge \frac{2}{\cos\lambda}\left(
\exp(-2\pi TN^{-n}\cos\lambda)-\exp(-2\pi TN^{j-n}\cos\lambda)
\right)-\Psi(t_{2m+1})+\Psi(t_{2m-1}).
\end{align*}
We now apply the elementary inequalities
$$
1-x\le e^{-x}\le 1-x+\frac{x^2}2,\quad x\in[0,+\infty),
$$
to obtain
\begin{align*}
&\exp(-2\pi TN^{-n}\cos\lambda)-\exp(-2\pi TN^{j-n}\cos\lambda) \\
&\ge 2\pi TN^{j-n}\cos\lambda-2\pi TN^{-n}\cos\lambda
-2\pi^2T^2N^{2j-2n}\cos^2\lambda.
\end{align*}
Since the intervals $(t_{2m-1},t_{2m+1})$ do not overlap and
are contained in $(t_1,t_1+2\pi)$ for $1\le m\le N^n,$ we obtain
\begin{align*}
&~ \quad L(r,f ) \\
&\ge \sum_{j=0}^{n-1}\sum_{m\in I_j}
\int_{t_{2m-1}}^{t_{2m+1}}|\gamma'(t)|dt \\
&\ge 2\pi T\sum_{j=0}^{n-1} N^{n-j}(1-N\inv)\left(
2N^{j-n}-2N^{-n}-2\pi TN^{2j-2n}\cos\lambda\right)
-\Psi(t_1+2\pi)+\Psi(t_1) \\
&=2\pi T(1-N\inv)\left(2n-\frac{2-2N^{-n}}{1-N\inv}
-2\pi TN\inv\frac{1-N^{-n}}{1-N\inv}\cos\lambda
\right)-2\pi \\
&> 4\pi T(1-N\inv)n-4\pi T(1+\pi TN\inv\cos\lambda)-2\pi.
\end{align*}
We now see that the choice $c_1=4\pi T(1+\pi T)+2\pi$ works.
 \end{pf}

\begin{lem}\label{lem:2}
The following inequality holds:
$$
\arth r_n\le  \frac{\pi T(1-N\inv)+c_2}{2\cos\lambda}\cdot n+c_3,
$$
where $c_2=c_2(\lambda, N)$ is a constant depending only on $\lambda$ and $N$
and $c_3=c_3(\lambda, N, T)$ is a constant depending only on $\lambda, N$ and $T.$
\end{lem}

In order to show the lemma, we recall the notion of hyperbolic metric.
For details, see \cite{KL:hg} for instance.
Let $\rho_V(z)|dz|$ denote the hyperbolic metric 
(of constant Gaussian curvature $-4$)
on a subdomain $V$ in $\C$ whose boundary contains at least two points.
It induces the hyperbolic distance
$$
\rho_V(z_1,z_2)=\inf_\gamma \int_\gamma \rho_V(z)|dz|,
$$
where the infimum is taken over all rectifiable curves $\gamma$
joining $z_1$ and $z_2$ in $V.$
For instance, when $V$ is the unit disk $\D,$ it is well known that
$$
\rho_\D(z_1,z_2)=\arth\left|\frac{z_1-z_2}{1-\overline{z_1}z_2}\right|,
\quad \arth x=\frac12\log\frac{1+x}{1-x}.
$$
A remarkable property of the hyperbolic distance is contractivity for holomorphic maps; 
namely, $\rho_V(f(z_1), f(z_2))\le \rho_W(z_1,z_2)$ for a holomorphic map $f:W\to V$ and for $z_1,z_2\in W.$
In particular, $\rho_V(z_1,z_2)\le \rho_W(z_1,z_2)$ whenever
$z_1,z_2\in W\subset V.$

\begin{pf}[Proof of Lemma \ref{lem:2}]
Without loss of generality, we may assume that $\lambda\ge 0.$
By the definition of $r_n,$ we have $r_n=|z_{n,m}|$ for some integer
$m=m_n$ with $0\le m\le N^n-1.$
We now choose integers $m_{n-1},\dots, m_0$ inductively as follows.
Suppose that $m_{j}$ was chosen so that $0\le m_{j}\le N^{j}-1.$
Let $m_{j-1}$ be an integer $m$ such that $|m+\frac12-(m_{j}+\frac12)/N|$ is minimal.
By the choice, we have 
$$
\frac1{2N}-1\le m_{j-1}-\frac{m_{j}}N\le \frac1{2N}
$$
and thus $0\le m_j-Nm_{j-1}<N.$
Let $l_j$ be the integer $m_j-Nm_{j-1}.$
Note that $m_0=0$ and $\kappa_0=\zeta_{0,0}.$
Then, for $m=m_n,$
\begin{align*}
\arth r_n=\rho_\D(0,z_{n,m})
&\le \rho_\D(0,z_{0,0})+\rho_\D(z_{0,0},z_{n,m}) \\
&=\rho_\Omega(0,w_{0,0})+\rho_\Omega(w_{0,0}, w_{n,m}) \\
&\le c_3+\rho_{\Omega^*}(\Pi(\kappa_0), \Pi(\kappa_n)) \\
&\le c_3+\sum_{j=1}^n \rho_{U}(\kappa_{j-1}, \kappa_j).
\end{align*}
where $c_3=\rho_\Omega(0,w_{0,0}).$

We next estimate $\rho_{U}(\kappa_{j-1}, \kappa_j).$
Fix $j$ with $1\le j\le n$ and set $$\tau_{j-1}=[(m_{j-1}+1/2)+T'\mu]/N^{j-1},$$ where
$T'=T/N+1/2.$
Then $\rho_{U}(\kappa_{j-1}, \kappa_j)\le \rho_{U}(\kappa_{j-1}, \tau_{j-1})
+\rho_{U}(\tau_{j-1}, \kappa_j).$
Let $$U_j=\{\zeta\in U: m_jN^{-j}<\Re[e^{-i\lambda}\zeta]<(m_j+1)N^{-j}\}$$ for $j=0,1,\dots, n.$
By the self-similarity of the domain $U,$ the function $u(\zeta)=N^{j-1}\zeta-m_{j-1}$
maps $U_{j-1}$ onto $U_{0}.$
It is simple to check that $u(\kappa_{j-1})=1/2+T\mu=\kappa_0,$
$u(\tau_{j-1})=1/2+T'\mu=\tau_0,$ and
$u(\kappa_j)=(l_j+1/2+T\mu)/N=\zeta_{1,l_j}.$
Then
$$
\rho_U(\kappa_j,\tau_{j-1})\le \rho_{U_{j-1}}(\kappa_j,\tau_{j-1})
=\rho_{U_0}(u(\kappa_j), u(\tau_{j-1}))
=\rho_{U_0}(\zeta_{1,l_j}, \tau_0).
$$
We want a bound for $\rho_{U_0}(\zeta_{1,l_j}, \tau_0)$ independent of $T.$
To this end, we further consider the map $v(\zeta)=\zeta+T\mu/N.$
Observe that the disk $|\zeta-\zeta_{1,l}|<(\cos\lambda)/2N$ is contained in $U_0$ for each $l=0,1,\dots, N-1,$
by the assumption $T(1-N\inv)\ge 1/2.$
Thus the two points $\zeta_{1,l}$ and $\tau_0$ are both contained in $v(W_l)\subset U_0,$ where 
$$
W_l=\{\zeta: 0<\Re[e^{-i\lambda}\zeta]<1, \Im\zeta>0\}
\cup\{\zeta: |\zeta-\tfrac{l+1/2}N|<\tfrac{\cos\lambda}{2N}\}.
$$
Hence,
$$
\rho_{U_0}(\zeta_{1,l}, \tau_0)\le
\rho_{v(W_l)}(\zeta_{1,l}, \tau_0)
=\rho_{W_l}(\tfrac{l+1/2}N, \tfrac{1+\mu}2)\le c, 
$$
where $\displaystyle c=\max_{l=0,1,\dots,N-1}\rho_{W_l}(\tfrac{l+1/2}N, \tfrac{1+\mu}2)$
is a constant depending only on $N$ and $\lambda.$

Similarly, we have $\rho_{U}(\kappa_{j-1}, \tau_{j-1})
\le\rho_{U_{j-1}}(\kappa_{j-1}, \tau_{j-1})=\rho_{U_0}(\kappa_0,\tau_0).$
The function $\varphi(z)=\sin\pi(z-1/2)$ maps the domain
$\Delta=\{z: 0<\Re z<1, \Im z>0\}$ conformally onto the upper half-plane $\uhp.$
We now consider the function 
$$h(z)=ze^{i\lambda}\cos\lambda+T\mu/N=e^{i\lambda}(z\cos\lambda+iT/N).$$
By the construction of the domain $U$ and the assumption $\lambda\ge0,$
we see that $h(\Delta)\subset U_0.$
Note that 
$$
h\inv(\tau_0)=\frac{\tau_0-T\mu/N}{e^{i\lambda}\cos\lambda}
=\frac{1+ie^{i\lambda}}{2e^{i\lambda}\cos\lambda}
=\frac12+i\cdot\frac{1-\sin\lambda}{2\cos\lambda}
=:\frac12+iy_0
$$
and 
$$
h\inv(\kappa_0)=\frac{T(1-N\inv)\mu+1/2}{e^{i\lambda}\cos\lambda}
=\frac12+i\cdot\left(\frac{T(1-N\inv)}{\cos\lambda}-\frac{\tan\lambda}2\right)
=:\frac12+iy_1
$$
belong to $\Delta.$
Therefore,
\begin{align*}
\rho_{U_0}(\kappa_0,\tau_0)&\le \rho_{h(\Delta)}(\kappa_0,\tau_0) \\
&=\rho_\Delta(1/2+iy_1,1/2+iy_0) \\
&=\rho_{\uhp}(\varphi(1/2+iy_1),\varphi(1/2+iy_0)) \\
&=\rho_{\uhp}(i\sinh \pi y_1, i\sinh \pi y_0) \\
&=\frac12\log\frac{\sinh\pi y_1}{\sinh\pi y_0} \\
&\le\frac{\pi(1-N\inv)}{2\cos\lambda}T+c',
\end{align*}
where $c'$ is a constant depending only on $\lambda.$
Summing up the above inequalities gives us
$$
\arth r_n\le c_3+\sum_{j=1}^n \big[\rho_U(\kappa_j,\tau_{j-1})+
\rho_U(\tau_{j-1},\kappa_{j-1})\big]
\le c_3+\left(\frac{\pi(1-N\inv)T}{2\cos\lambda}+c'+c\right)n,
$$
which proves the required assertion with $c_2=c+c'.$
\end{pf}

We are now ready to show our second theorem.

\begin{pf}[Proof of Theorem \ref{thm:main2}]
We first note that $M(r,f)\to 1$ as $r\to1^-$ 
for the function $f$ constructed above.
By Lemmas \ref{lem:1} and \ref{lem:2}, we have
$$
\frac{L(r_n,f)}{\arth r_n}\ge\frac{4\pi T(1-N\inv)n-c_1(T)}%
{\frac{\pi T(1-N\inv)+c_2(\lambda, N)}{2\cos\lambda}\cdot n+c_3(\lambda, N, T)}.
$$
Hence,
$$
\liminf_{n\to\infty}\frac{L(r_n,f)}{\arth r_n}
\ge \frac{8\pi T(1-N\inv)\cos\lambda}{\pi T(1-N\inv)+c_2(\lambda, N)}.
$$
We now choose $T$ so large that the right-hand side is bigger
than $(8-2\varepsilon)\cos\lambda.$
Consequently,
$$
\liminf_{r\to 1^-}\frac{L(r,f)}{M(r,f)\log\frac{1+r}{1-r}}
=\liminf_{r\to 1^-}\frac{L(r,f)}{2\,\arth r}
\ge (4-\varepsilon)\cos\lambda
$$
as required.
\end{pf}

\section{Concluding remarks}

We end this paper with a couple of remarks.
As we mentioned above, the example $f$ constructed in the previous section is
exactly same as the function constructed by Crane and Markose \cite{CM05}
when $\lambda=0.$
In this case, they observed in \cite[Lemma 3]{CM05} a self-similar property of $f.$
Let us see a similar property of our function $f$ here.
We will use the same notation as in the previous section.

Since the image domain of $U$ under the mapping $\zeta\mapsto\zeta/N$ is
$$
\tilde U=\uhp\setminus \left( A_0'\cup \bigcup_{n=1}^\infty A_n\right),
$$
where
$$
A_0'=\bigcup_{m\in\Z}(m,\alpha_{1,m}].
$$
We set $\Omega'=\Pi(U').$
Then the inverse image of $\Omega$ under the mapping $w\mapsto w^N$ is exactly $\Omega'.$
Note that $\Omega\subset\Omega'$ and $\Omega'\setminus\Omega=[a_{0,0},a_{1,0})_\lambda,$ part of the $\lambda$-spiral $[a_{0,0},0]_\lambda.$
Let $f_1(z)=f(z^N)^{1/N}$ be the analytic branch of the $N$-th root of $f(z^N)$
with $f_1'(0)>0$ (power roots of similar functions will be taken in the same way in the following).
Then $f_1$ is a conformal homeomorphism of $\D$ onto $\Omega'.$
We now observe that $\varphi=f_1\inv\circ f:\D\to\D$ is univalent and $\varphi(\D)=\D\setminus\gamma,$
where $\gamma=f_1\inv([a_{0,0},a_{1,0})_\lambda)$ is a slit emanating from a point
of $\partial\D.$
When $\lambda=0,$ we know that $\gamma$ is a line segment of the form $[\beta,1)$
for some $\beta\in(0,1)$ by the symmetry $\varphi(\bar z)=\overline{\varphi(z)}.$
However, when $\lambda\ne0$ we cannot detect the form of $\gamma$ at the moment.
Anyway, since $f_1(z)=f(z^N)^{1/N}=f(\varphi\inv(z)),$ $f$ has a self-similarity in the sense that
$f(w)^N=f(\varphi(w)^N)$ for $w=\varphi\inv(z).$
If we put $\psi(w)=\varphi(w)^N,$ we have the relation $f(w)^N=f(\psi(w)).$
By applying this repeatedly, we have $f(w)^{N^n}=f(\psi^{\circ n}(w)),$
where $\psi^{\circ n}$ means the $n$-times iteration of $\psi.$
If we expand $f$ in the form $f(z)=a_1z+a_2z^2+\dots,$ we have
$$
f(z^{N^n})^{N^{-n}}=a_1^{N^{-n}}z\left(1+\frac{a_2z^{N^n}}{a_1N^n}+\dots\right).
$$
We note that $(\psi^{\circ n}(w))^{N^{-n}}/w\to 1$ as $n\to\infty.$
Hence, as in \cite{CM05}, we obtain the expression
$$
f(w)=f(\psi^{\circ n}(w))^{N^{-n}}=\lim_{n\to\infty}\psi^{\circ n}(w)^{N^{-n}}.
$$

We finally mention a relation between $\lambda$-spirallike functions and starlike functions
(see \cite{KS12spl} for some details).
For a $\lambda$-spirallike function $f$ on $\D$ with $f(0)=f'(0)-1=0,$
there is a starlike function $g$ on $\D$ such that
$$
\frac{f(z)}{z}=\left(\frac{g(z)}{z}\right)^\alpha,
\quad \alpha=e^{i\lambda}\cos\lambda.
$$
It seems, however, that there is no simple relation between $L(r,f)$ and $L(r,g).$

\section{Conflict of interest}
The authors declare that there is no conflict of interest
regarding the publication of this paper.

\def\cprime{$'$} \def\cprime{$'$} \def\cprime{$'$}
\providecommand{\bysame}{\leavevmode\hbox to3em{\hrulefill}\thinspace}
\providecommand{\MR}{\relax\ifhmode\unskip\space\fi MR }
\providecommand{\MRhref}[2]{%
  \href{http://www.ams.org/mathscinet-getitem?mr=#1}{#2}
}
\providecommand{\href}[2]{#2}

\end{document}